\documentclass[final,1p,times]{elsarticle}

\usepackage{amsmath,amssymb,amsfonts,latexsym,enumerate}
\usepackage[T1]{fontenc}
\usepackage[cp1250]{inputenc}
\usepackage{color}

\newcommand{\E}{\mathbb{E}\,}

\renewcommand{\leq}{\leqslant}

\renewcommand{\geq}{\geqslant}
\newtheorem{theorem}{Theorem}[section]
 
 \newtheorem{lemma}[theorem]{Lemma}

 \newdefinition{remark}[theorem]{Remark}
 \newdefinition{example}[theorem]{Example}
\newdefinition{definition}[theorem]{Definition}

 \newproof{proof}{Proof}
 \newproof{pfs}{Proof of Theorem~\ref{th:support}}
 \newproof{pfd}{Proof of Lemma~\ref{lm:delta}}

\newcommand{\N}{\mathbb{N}}

\newcommand{\lcx}{\leq_{\text{\rm cx}}}

\newcommand{\lst}{\leq_{\text{\rm st}}}
\newcommand{\gst}{\geq_{\text{\rm st}}}

\newcommand{\C}{\mathcal{C}}

\numberwithin{equation}{section}

\journal{\phantom{.}}

\begin{document}

\begin{frontmatter}



\title{Muirhead inequality for convex orders and a problem of I.\ Ra\c{s}a on Bernstein polynomials}

\author{Andrzej Komisarski}
\ead{andkom@math.uni.lodz.pl}
\address{Department of Probability Theory and Statistics, Faculty of Mathematics and Computer Science,
University of \L\'od\'z, ul. Banacha 22, 90-238 \L\'od\'z, Poland}
\author{Teresa Rajba}
\ead{trajba@ath.bielsko.pl}
\address{Department of Mathematics, University of Bielsko-Biała, ul. Willowa 2, 43-309 Bielsko-Biała, Poland}

\begin{abstract}

We present a new, very short proof of a conjecture by I.\ Ra\c{s}a, which is an inequality involving basic Bernstein polynomials and convex functions.
It was affirmed positively very recently by J.\ Mrowiec, T.\ Rajba and S.\ Wąsowicz (2017) by the use of stochastic convex orders,
as well as by Abel (2017) who simplified their proof. We give a useful sufficient condition for the verification
of some stochastic convex order relations, which in the case of binomial distributions are equivalent to the I.\ Ra\c{s}a inequality. We give also the corresponding inequalities for other distributions.
Our methods allow us to give some extended versions of stochastic convex orderings as well as the I.\ Ra\c{s}a type inequalities.
In particular, we prove the Muirhead type inequality for convex orders for convolution polynomials of probability distributions.
\end{abstract}


\end{frontmatter}


\section{Introduction}
For $n\in\N$ and $i=0,1,\dots,n$, let $b_{n,i}(x)=\binom{n}{i}x^i(1-x)^{n-i}$ ($x\in[0,1]$)
denote the Bernstein basic polynomials.
The Bernstein operator $B_n:\C([0,1])\to\C([0,1])$ (\cite{Adell1996, Adell1999}) associates to each continuous function $\varphi:[0,1]\to\mathbb R$ the function
$B_n(\varphi):[0,1]\to\mathbb R$ given by
$B_n(\varphi)(x)=\sum_{i=0}^n b_{n,i}(x)\varphi\left(\frac in\right)$.
Recently, J.\ Mrowiec, T.\ Rajba and S.\ Wąsowicz \cite{MRW2017} proved the following theorem on inequality for Bernstein operators:
\begin{theorem}\label{th:1}
Let $n\in\N$ and $x,y\in[0,1]$. Then
\begin{equation}\label{eq:Rasa}
 \sum_{i=0}^n \sum_{j=0}^n \left(b_{n,i}(x)b_{n,j}(x)+b_{n,i}(y)b_{n,j}(y)-2b_{n,i}(x)b_{n,j}(y)\right)\varphi\left(\tfrac{i+j}{2n}\right)\geq0
\end{equation}
for all convex functions $\varphi\in\C([0,1])$.
\end{theorem}

This inequality involving Bernstein basic polynomials and convex functions was conjectured as an open problem 25 years ago by I.\ Ra\c{s}a.
J.\ Mrowiec, T.\ Rajba and S.\ Wąsowicz \cite{MRW2017} showed that the conjecture is true. Their proof makes heavy use of probability theory.
As a tool they applied a concept of stochastic convex orders (which they proved for binomial distributions) as well as the so-called concentration inequality. Later
U.\ Abel \cite{Abel2016} gave an elementary proof of the above theorem, which was much shorter than that given in \cite{MRW2017}.
In this paper, we present a new proof of the above theorem, which is significantly simpler and shorter than that given by U.\ Abel \cite{Abel2016}
(cf. Theorem \ref{th:condition} and Theorem \ref{th:Rasa}a). As a tool we use both stochastic convex orders as well as the usual stochastic order.

In \cite{Abel2016}, U.\ Abel studied Mirakyan--Sz\'asz operators $S_n:D_S\to\C([0,\infty))$ (where $D_S\subset\C([0,\infty))$
consists of functions of at most exponential growth) given by
$S_n(\varphi)(x)=\sum_{i=0}^{\infty}s_{i}(nx)\varphi\left(\tfrac in\right)$,
where $s_{i}(x)=e^{-x}\cdot\frac{x^i}{i!}$ are the corresponding basic functions,
and he proved the following inequality (corresponding to the inequality \eqref{eq:Rasa}) for these operators and convex functions $\varphi\in D_S$
\begin{equation}\label{eq:Rasabis}
 \sum_{i=0}^\infty\sum_{j=0}^\infty\left(s_i(x)s_j(x)+s_i(y)s_j(y)-2s_{i}(x)s_{j}(y)\right)\varphi\left(\tfrac{i+j}{2n}\right)\geq0\quad\text{for }x,y\in[0,\infty).
\end{equation}

U.\ Abel \cite{Abel2016} considered also Baskakov operators $V_n:D_{V_n}\to\C([0,\infty))$
(where $D_{V_n}\subset\C([0,\infty))$) given by
$V_n(\varphi)(x)=\sum_{i=0}^{\infty}v_{n,i}(x)\varphi\left(\tfrac in\right)$,
where $v_{n,i}(x)=\binom{n+i-1}i\cdot\frac{x^i}{(1+x)^{n+i}}$ (with $x\in[0,\infty)$) are the Baskakov basic functions,
and proved the corresponding inequality
\begin{equation}\label{eq:baskRasa}
 \sum_{i=0}^{\infty} \sum_{j=0}^{\infty} \left(v_{n,i}(x)v_{n,j}(x)+v_{n,i}(y)v_{n,j}(y)-2v_{n,i}(x)v_{n,j}(y)\right)\varphi\left(\tfrac{i+j}{2n}\right)\geq0\quad\text{for }x,y\in[0,\infty).
\end{equation}

Note, that the basic functions, which appear in the formulas for the operators $B_n$, $S_n$ and $V_n$, are the probabilities
of binomial, Poisson, and negative binomial distributions, respectively.
In this paper, we study some other families of probability distributions that can be used as basic functions for operators and inequalities associated with these operators.

In Section 2, we prove Theorem \ref{th:condition} on convex orders, which is a useful tool for proving
inequalities \eqref{eq:Rasa}, \eqref{eq:Rasabis}, \eqref{eq:baskRasa} and many similar inequalities
(we call them I.\ Ra\c{s}a type inequalities).
In particular, using this theorem for binomial distributions, Poisson distributions and negative binomial distributions,
we obtain new, very short proofs of inequalities \eqref{eq:Rasa}, \eqref{eq:Rasabis} and \eqref{eq:baskRasa}.

The Bernstein-Schnabl operators $B_n$, Mirakyan-Sz\'asz operators $S_n$ and Baskakov operators $V_n$
are Markov operators and they share the following property: if $T$ is any of these operators and function
$\varphi$ is affine, then $T(\varphi)=\varphi$. This nice property has been investigated by many authors.
In particular, in \cite{Rasa2016, Rasa2014, Rasa2017} the authors discussed the theory of Markov operators
and partial differential equations related to Markov operators.
I.\ Ra\c{s}a's conjecture (Theorem~\ref{th:1}) and other I.\ Ra\c{s}a type inequalities
have their consequences in this theory (see \cite{Rasa2016, Rasa2014, Rasa2017} for the details).
In Remark~\ref{th:remark} we consider the other examples of such operators.

In Section 3, we give a strong generalization of Theorem \ref{th:condition} and use it
to obtain subsequent generalizations of the inequality \eqref{eq:Rasa}.
The main result of Section 3 (Theorem \ref{th:main}) is the Muirhead type inequality for convex orders
for convolution polynomials of probability distributions.

\section{The I.\ Ra\c{s}a type inequalities}

In the sequel we make use of the theory of stochastic orders. Let us recall some basic notations and results (see \cite{Shaked2007}).
Let $\mu$ be a probability distribution (a Borel measure on $\mathbb R$ satisfying $\mu(\mathbb R)=1$).
For $x\in\mathbb R$ the delta symbol $\delta_x$ denotes one-point probability distribution satisfying $\delta_x(\{x\})=1$.
As usual, $F_\mu(x)=\mu((-\infty,x])$ ($x\in\mathbb R$) stands for the cumulative distribution function of $\mu$.
If $\mu$ and $\nu$ are two probability distributions such that
$F_\mu(x)\geq F_\nu(x)$ for all $x\in\mathbb R$,
then $\mu$ is said to be \emph{smaller than $\nu$ in the~usual stochastic order} (denoted by $\mu\lst\nu$).
An important characterization of the~usual stochastic order is the following theorem.
\begin{theorem}[\cite{Shaked2007}, p. 5]\label{th:1a1}
Two probability distributions $\mu$ and $\nu$ satisfy $\mu\lst\nu$ if, and only if there exist two random variables $X$ and $Y$
defined on the same probability space, such that $X\sim\mu$, $Y\sim\nu$ and $P(X\leq Y)=1$.
\end{theorem}


If $\mu$ and $\nu$ are two two probability distributions such that
$$\int\varphi(x)\mu(dx)\leq\int\varphi(x)\nu(dx) \quad \text{for all convex functions }\ \varphi\colon\mathbb R\to\mathbb R,$$
provided the integrals exist, then $\mu$ is said to be \emph{smaller than $\nu$ in the convex stochastic order} (denoted as $\mu\lcx\nu$).

We will need the following special case of
the Hardy-Littlewood-P\'olya inequality (\cite{Hardy1952}, Theorem 108):
\begin{remark}\label{th:HLP}
Let $E\subset\mathbb R$ be a convex subset of the real line and let $\varphi:E\to\mathbb R$ be a convex function.
If $a\leq b,c\leq d$ are in $E$ and $a+d=b+c$, then $\varphi(b)+\varphi(c)\leq\varphi(a)+\varphi(d)$.
\end{remark}

In the following theorem, we give a very useful sufficient condition that will be used for the verification of some convex stochastic orders.
\begin{theorem}\label{th:condition}
Let $\mu$ and $\nu$ be two probability distributions with finite first moments, such that $\mu\lst\nu$ or $\nu\lst\mu$. Then
\begin{equation}\label{eq:main}
\mu*\nu\lcx\tfrac12(\mu*\mu+\nu*\nu).
\end{equation}
\end{theorem}

\begin{proof}
Without loss of generality we may assume that $\mu\lst\nu$.
Then, by Theorem \ref{th:1a1}, there exist two independent random vectors $(X_1,Y_1)$ and $(X_2,Y_2)$ such that 
\begin{equation}\label{eq:suf3}
X_1,X_2\sim\mu, \quad Y_1,Y_2\sim\nu, \quad P(X_1\leq Y_1)=1 \quad \text{and} \quad P(X_2\leq Y_2)=1.
\end{equation}
Then we have $X_1+Y_2\sim\mu*\nu$, $X_2+Y_1\sim\mu*\nu$, $X_1+X_2\sim\mu*\mu$
and $Y_1+Y_2\sim\nu*\nu$.
By \eqref{eq:suf3}, $P(X_1+X_2\leq X_1+Y_2\leq Y_1+Y_2)=1$, $P(X_1+X_2\leq X_2+Y_1\leq Y_1+Y_2)=1,$
and obviously $P((X_1+Y_2)+(X_2+Y_1)=(X_1+X_2)+ (Y_1+Y_2))=1$.
By Remark \ref{th:HLP}, we conclude that
$P(\varphi(X_1+Y_2)+\varphi(X_2+Y_1)\leq \varphi(X_1+X_2)+ \varphi(Y_1+Y_2))=1$
for all convex functions $\varphi\colon\mathbb R\to\mathbb R$, which implies
\begin{multline*}
\int\varphi(x)(\mu*\nu)(dx)=\tfrac12\E\left((\varphi(X_1+Y_2)+\varphi(X_2+Y_1)\right)\leq\\\\
\tfrac12\E\left(\varphi(X_1+X_2)+\varphi(Y_1+Y_2))\right)
=\tfrac12\left(\int\varphi(x)(\mu*\mu)(dx)+\int\varphi(x)(\nu*\nu)(dx)\right).
\end{multline*}
Thus $\mu*\nu\lcx\tfrac12(\mu*\mu+\nu*\nu)$.
\end{proof}


The following example shows that the converse of Theorem \ref{th:condition} is not true.

\begin{example}
If $\mu=\frac12\delta_{-3}+\frac12\delta_1$ and $\nu=\frac34\delta_0+\frac14\delta_4$, then \eqref{eq:main} holds, although neither the condition $\mu\lst\nu$ nor $\nu\lst\mu$ is satisfied (we leave the proof to the reader).
\end{example}

We will apply Theorem~\ref{th:condition} for $\mu$ and $\nu$ from various families of probability distributions.
As a~result we obtain new proofs of the results of J.\ Mrowiec, T.\ Rajba and S.\ Wąsowicz \cite{MRW2017} (Theorem \ref{th:1} and \cite{MRW2017}),
U.\ Abel (\cite{Abel2016}) and several new inequalities, which are analogues of \eqref{eq:Rasa}.

The binomial distribution with parameters $n\in\N$ and $p\in[0,1]$ (denoted by $B(n,p)$)
is the probability distribution given by
$B(n,p)(\{k\})=b_{n,k}(p)=\binom nkp^k(1-p)^{n-k}$ for $k=0,1,\dots,n$
and $B(n,p)(\mathbb R\setminus\{0,1,\dots\,n\})=0$.

The Poisson distribution with the parameter $\lambda>0$ (denoted by $Poiss(\lambda)$)
is the probability distribution given by
$Poiss(\lambda)(\{i\})=s_i(\lambda)=e^{-\lambda}\cdot\tfrac{\lambda^i}{i!}$ for $i=0,1,\dots$
and $Poiss(\lambda)(\mathbb R\setminus\{0,1,\dots\})=0$.
By convention, we say that $Poiss(0)=\delta_0$ (i.e. $s_0(0)=1 $ and $s_i(0)=0$ for $i=1,2,\dots$).

The negative binomial distribution with parameters $r>0$ and $0\leq p<1$ (denoted by $NB(r,p)$)
is the probability distribution given by
$$NB(r,p)(\{k\})=nb_k(r,p)=\binom{k+r-1}kp^k(1-p)^r=\frac{\Gamma(k+r)}{\Gamma(r)\cdot k!}p^k(1-p)^r\quad \text{for}\quad k=0,1,\dots$$
and $NB(r,p)(\mathbb R\setminus\{0,1,\dots\})=0$.
By convention, we say that if $0\leq p<1$, then $NB(0,p)=\delta_0$ (i.e., $nb_0(0,p)=1$ and $nb_k(0,p)=0$ for $k>0$).
The geometric distribution is a special case of the negative binomial distribution, namely $Geom(p)=NB(1,1-p)$.

The gamma distribution with shape $\alpha>0$ and rate $\beta>0$ (denoted by $\Gamma(\alpha,\beta)$)
is the distribution given by the density function
$\gamma_{\alpha,\beta}(x)=\frac{\beta^{\alpha}x^{\alpha-1}e^{-x\beta}}{\Gamma(\alpha)}$ for $x>0$
and $\gamma_{\alpha,\beta}(x)=0$ for $x\leq0$.
By convention, we define $\Gamma(0,\beta)=\delta_0$ for every $\beta>0$.

The beta distribution with parameters $\alpha,\beta>0$ (denoted $Beta(\alpha,\beta)$)
is the distribution given by the density function
$$\overline b_{\alpha,\beta}(x)=\frac{\Gamma(\alpha+\beta)}{\Gamma(\alpha)\Gamma(\beta)}x^{\alpha-1}(1-x)^{\beta-1} \quad \text{for }x\in(0,1),$$
and $\overline b_{\alpha,\beta}(x)=0$ for $x\notin(0,1)$.
By convention, we define $Beta(0,\beta)=\delta_0$ and $Beta(\alpha,0)=\delta_1$ for every $\alpha,\beta>0$.

By $f_{m,\sigma^2}$ we denote the density function of $N(m,\sigma^2)$, the normal (Gaussian) distribution
with the mean $m\in\mathbb R$ and the variance $\sigma^2>0$.

\begin{lemma}\label{th:rozklady}
\begin{itemize}
\item[{\rm\textbf{a)}}]
Let $n\in\mathbb N$ and $p_1,p_2\in[0,1]$. Then
$B(n,p_1)\lst B(n,p_2)\ \Leftrightarrow\ p_1\leq p_2$.
\item[{\rm\textbf{b)}}]
Let $\lambda_1,\lambda_2\geq 0$. Then
$Poiss(\lambda_1)\lst Poiss(\lambda_2)\ \Leftrightarrow\ \lambda_1\leq\lambda_2$.
\item[{\rm\textbf{c)}}]
Let $r_1,r_2\geq0$ and $p_1,p_2\in[0,1)$.
If $r_1\leq r_2$ and $p_1\leq p_2$, then $NB(r_1,p_1)\lst NB(r_2,p_2)$.
\item[{\rm\textbf{d)}}]
Let $\alpha_1,\alpha_2\geq0$ and $\beta_1,\beta_2>0$.
If $\alpha_1\leq\alpha_2$ and $\beta_1\geq\beta_2$, then $\Gamma(\alpha_1,\beta_1)\lst\Gamma(\alpha_2,\beta_2)$.
\item[{\rm\textbf{e)}}]
Let $\alpha_1,\beta_1,\alpha_2,\beta_2\geq0$ be such that $\alpha_1+\beta_1>0$ and $\alpha_2+\beta_2>0$.
If $\alpha_1\leq\alpha_2$ and $\beta_1\geq\beta_2$, then $Beta(\alpha_1,\beta_1)\lst Beta(\alpha_2,\beta_2)$.
\item[{\rm\textbf{f)}}]
Let $m_1,m_2\in\mathbb R$ and $\sigma_1^2,\sigma_2^2>0$. Then
$N(m_1,\sigma_1^2)\lst N(m_2,\sigma_2^2)$ $\Leftrightarrow$ $(m_1\leq m_2$ and $\sigma_1^2=\sigma_2^2)$.
\end{itemize}
\end{lemma}

\begin{proof}
For \textbf{a)} and \textbf{f)} see \cite{Shaked2007}, p. 14

\textbf{b)} It is enough to show $\lambda_1\leq\lambda_2\Rightarrow Poiss(\lambda_1)\lst Poiss(\lambda_2)$.
Assume that $0\leq\lambda_1\leq\lambda_2$. Let $X\sim Poiss(\lambda_1)$ and $Z\sim Poiss(\lambda_2-\lambda_1)$
be independent random variables. Then $Y:=X+Z$ satisfies $Y\sim Poiss(\lambda_2)$ and $P(X\leq Y)=P(Z\geq0)=1$.
Hence Theorem \ref{th:1a1} yields $Poiss(\lambda_1)\lst Poiss(\lambda_2)$.

\textbf{c)} We shall use the following observation: Let $r>0$ and $p\in[0,1)$. If $(N_t)_{t\geq0}$ is the Poisson process with intensity $\lambda=1$
and $T\sim\Gamma(r,1)$ is independent of $(N_t)_{t\geq0}$, then $N_{\frac p{1-p}\cdot T}\sim NB(r,p)$.
Indeed, for $k=0,1,\dots$ we have
\begin{multline*}
 P(N_{\frac p{1-p}\cdot T}=k)=\int_0^\infty\frac{t^{r-1}e^{-t}}{\Gamma(r)}\cdot\frac{\left(\frac p{1-p}\cdot t\right)^k}{k!}e^{-\frac p{1-p}\cdot t}dt=\\\\
 \frac{\Gamma(k+r)}{\Gamma(r)\cdot k!}p^k(1-p)^r\cdot\int_0^\infty\frac{\left(\frac 1{1-p}\right)^{k+r}t^{k+r-1}e^{-\frac 1{1-p}\cdot t}}{\Gamma(k+r)}dt=
 \binom{k+r-1}kp^k(1-p)^r\cdot1.
\end{multline*}
Assume that $0<r_1\leq r_2$ and $0\leq p_1\leq p_2<1$, thus $0\leq\frac{p_1}{1-p_1}\leq\frac{p_2}{1-p_2}$.

Let $(N_t)_{t\geq0}$ (the Poisson process with intensity $\lambda=1$), $T\sim\Gamma(r_1,1)$ and $Z\sim NB(r_2-r_1,p_2)$ be independent.
We set $X:=N_{\frac{p_1}{1-p_1}\cdot T}$ and $Y:=N_{\frac{p_2}{1-p_2}\cdot T}+Z.$
Then $X\sim NB(r_1,p_1)$, $Y\sim NB(r_2,p_2)$ and
$$P(X\leq Y)=P\left(\tfrac{p_1}{1-p_1}\cdot T\leq\tfrac{p_2}{1-p_2}\cdot T\text{ and }Z\geq0\right)=1.$$
Hence Theorem \ref{th:1a1} yields $NB(r_1,p_1)\lst NB(r_2,p_2)$.

\textbf{d)} Assume that $0\leq\alpha_1\leq\alpha_2$ and $\beta_1\geq\beta_2>0$, thus $\frac{\beta_1}{\beta_2}\geq1$.
Let $X\sim\Gamma(\alpha_1,\beta_1)$ and $Z\sim\Gamma(\alpha_2-\alpha_1,\beta_1)$
be independent random variables. Then $X+Z\sim\Gamma(\alpha_2,\beta_1)$ and $Y:=\frac{\beta_1}{\beta_2}\cdot(X+Z)$
satisfies $Y\sim\Gamma(\alpha_2,\beta_2)$.
Moreover, $P(X\leq Y)\geq P(Z\geq0$ and $X+Z\geq0)=1$. Hence Theorem \ref{th:1a1} yields $\Gamma(\alpha_1,\beta_1)\lst\Gamma(\alpha_2,\beta_2)$.

\textbf{e)} Let $U\sim\Gamma(\alpha_1,1)$, $V\sim\Gamma(\alpha_2-\alpha_1,1)$, $W\sim\Gamma(\beta_2,1)$ and $Z\sim\Gamma(\beta_1-\beta_2,1)$
be independent random variables.
We set $X:=\frac{U}{U+W+Z}$ and $Y:=\frac{U+V}{U+V+W}$.
Then $X\sim Beta(\alpha_1,\beta_1)$, $Y\sim Beta(\alpha_2,\beta_2)$ and $P(X\leq Y)=1$.
Theorem \ref{th:1a1} yields $Beta(\alpha_1,\beta_1)\lst Beta(\alpha_2,\beta_2)$.
\end{proof}


By Lemma \ref{th:rozklady}, as an immediate consequence of Theorem \ref{th:condition}, we obtain the following I.\ Ra\c{s}a type inequalities:

\begin{theorem}\label{th:Rasa}
\begin{itemize}
\item[{\rm\textbf{a)}}]
If $n\in\N$ and $x,y\in[0,1]$, then \eqref{eq:Rasa} holds for all convex functions $\varphi\in\C([0,1])$.
\item[{\rm\textbf{b)}}]
If $x,y\geq0$, then \eqref{eq:Rasabis} holds for all convex functions $\varphi\colon[0,\infty)\to\mathbb R$.
\item[{\rm\textbf{c)}}]
Let $r_1,r_2>0$ and $p_1,p_2\in[0,1)$. If $(r_1-r_2)(p_1-p_2)\geq0$, then
{\small\begin{equation}\label{eq:NB6}
\sum_{i=0}^\infty\sum_{j=0}^\infty\left[nb_i(r_1,p_1)\cdot nb_j(r_1,p_1)+
nb_i(r_2,p_2)\cdot nb_j(r_2,p_2)-
2\cdot nb_i(r_1,p_1)\cdot nb_j(r_2,p_2)\right]\cdot\varphi\left(\tfrac{i+j}2\right)\geq0
\end{equation}}
for all convex functions $\varphi\colon[0,\infty)\to\mathbb R$.
\item[{\rm\textbf{d)}}]
Let $\alpha_1,\alpha_2>0$ and $\beta_1,\beta_2>0$ satisfy $(\alpha_1-\alpha_2)(\beta_1-\beta_2)\leq0$. Then
{\small\begin{equation}\label{eq:gam12}
\int_0^{\infty}\int_0^{\infty}\left[\gamma_{\alpha_1,\beta_1}(u)\cdot\gamma_{\alpha_1,\beta_1}(v)+
\gamma_{\alpha_2,\beta_2}(u)\cdot\gamma_{\alpha_2,\beta_2}(v)
-2\cdot\gamma_{\alpha_1,\beta_1}(u)\cdot\gamma_{\alpha_2,\beta_2}(v)\right]\cdot\varphi\left(\tfrac{u+v}2\right)du\ dv\geq0
\end{equation}}
for all convex functions $\varphi\colon[0,\infty)\to\mathbb R$.
\item[{\rm\textbf{e)}}]
Let $\alpha_1,\beta_1,\alpha_2,\beta_2>0$ satisfy $(\alpha_1-\alpha_2)(\beta_1-\beta_2)\leq0$. Then
{\small\begin{equation}\label{eq:bet12}
\int_0^1\int_0^1\left[\overline b_{\alpha_1,\beta_1}(u)\cdot\overline b_{\alpha_1,\beta_1}(v)+
\overline b_{\alpha_2,\beta_2}(u)\cdot\overline b_{\alpha_2,\beta_2}(v)-
2\cdot\overline b_{\alpha_1,\beta_1}(u)\cdot\overline b_{\alpha_2,\beta_2}(v)\right]\cdot\varphi\left(\tfrac{u+v}2\right)du\ dv\geq0
\end{equation}}
for all convex functions $\varphi\colon[0,1]\to\mathbb R$.
\item[{\rm\textbf{f)}}]
If $\sigma^2>0$ and $m_1,m_2\in\mathbb R$, then
{\small$$\int_{-\infty}^\infty\int_{-\infty}^\infty\left[f_{m_1,\sigma^2}(u)\cdot f_{m_1,\sigma^2}(v)+f_{m_2,\sigma^2}(u)\cdot f_{m_2,\sigma^2}(v)-
2\cdot f_{m_1,\sigma^2}(u)\cdot f_{m_2,\sigma^2}(v)\right]\cdot\varphi\left(\tfrac{u+v}2\right)du\ dv\geq0$$}
for all convex functions $\varphi\colon\mathbb R\to\mathbb R$.
\end{itemize}
\end{theorem}

\begin{proof}
The proofs of a)--f) are very similar. We present the proof of a).
Let $x,y\in[0,1]$, $\mu=B(n,x)$ and $\nu=B(n,y)$. By Lemma \ref{th:rozklady}a,
$\mu\lst\nu$ or $\nu\lst\mu$.
By Theorem \ref{th:condition} we obtain $\mu*\nu\lcx\tfrac12(\mu*\mu+\nu*\nu)$.
It follows that for all convex functions $\varphi\in\C\left([0,1]\right)$ we have
\begin{multline*}
2\iint\varphi\left(\tfrac{u+w}{2n}\right)\mu(du)\ \nu(dw)=
2\int\varphi\left(\tfrac z{2n}\right)(\mu*\nu)(dz)\leq\\\\
\int\varphi\left(\tfrac z{2n}\right)(\mu*\mu)(dz)+\int\varphi\left(\tfrac z{2n}\right)(\nu*\nu)(dz)=
\iint\varphi\left(\tfrac{u+w}{2n}\right)\mu(du)\ \mu(dw)+\iint\varphi\left(\tfrac{u+w}{2n}\right)\nu(du)\ \nu(dw),
\end{multline*}
which can be rewritten in the form \eqref{eq:Rasa}. The theorem is proved.
\end{proof}


Theorem~\ref{th:Rasa}a is the original I.\ Ra\c sa conjecture (cf. Theorem \ref{th:1}) proved in \cite{MRW2017} .
Theorem~\ref{th:Rasa}b and the special case of Theorem~\ref{th:Rasa}c (when $r$ is a natural number) are proved in \cite{Abel2016}.
U.\ Abel \cite{Abel2016} gave elementary proofs of these inequalities, but the new proof given here is significantly simpler and shorter.
Applying \eqref{eq:NB6} with $r_1=r_2=r$, $p_1=\frac x{1+x}$ and $p_2=\frac y{1+y}$ (with $x,y\geq0$)
results in the inequality \eqref{eq:baskRasa}, which is associated with the Baskakov operator $V_r$.

Note that $r_1=r_2=1$ in Theorem~\ref{th:Rasa}c corresponds to the case of the geometric probability distribution.
Similarly, $\alpha_1=\alpha_2=1$ in Theorem~\ref{th:Rasa}d corresponds to the case of the exponential probability distribution.


In Theorem~\ref{th:Rasa}e it is worth to consider the reparametrization
$\alpha=xt$ and $\beta=(1-x)t$, where $t>0$ is fixed and $x\in[0,1]$.
The obtained I.\ Ra\c sa type inequality is related to the family of beta operators
$\mathcal B_t:\C([0,1])\to\C([0,1])$ defined by $\mathcal B_t(\varphi)(x)=\int_0^1\varphi(u)\overline b_{xt,(1-x)t}(u)du$.
Similar reparametrizations in Theorem~\ref{th:Rasa}a--f lead to I.\ Ra\c sa type inequalities related to the other families of operators:
Bleimann--Butzer--Hahn operators $L_n$, gamma operators $G_t$, M\"uller gamma operators $G_t^*$, Lupa\c s beta operators $\mathcal B_t^*$,
inverse beta operators $T_t$, Meyer-K\"onig--Zeller operators $M_t$. We recall that inequalities \eqref{eq:Rasa}, \eqref{eq:Rasabis} and \eqref{eq:baskRasa}
are related to Bernstein operators $B_n$, Mirakyan--Sz\'asz operators $S_t$ and Baskakov operators $V_t$, respectively.
For the definitions and properties of the above operators see e.g \cite{Adell1993} or \cite{Adell1996}.

All these operators are Markov operators
(we recall that $T:\C(K)\to\C(K)$ is a Markov operator if it is positive and $T(\varphi)=\varphi$, whenever $\varphi$ is constant).
Some of them ($B_n$, $G_t$, $S_t$, $\mathcal B_t$, $V_t$ and $M_t$) share the additional property: if $T$ is any of these operators,
then $T(\varphi)=\varphi$ for affine $\varphi$.
Consequently, they fit into the theory studied in \cite{Rasa2016, Rasa2017} and they form a new example of operators
satisfying some conditions discussed in \cite{Rasa2016, Rasa2017} (see \cite{Rasa2016}, Example 1.1 and conditions $(c_1)$ and $(c_2)$).
In view of \cite{Rasa2016, Rasa2017} Bernstein operators $B_n$ and beta operators $\mathcal B_t$
are especially interesting, because they share one more property:
they are the endomorphisms of $\C([0,1])$ and $[0,1]$ is compact.
It is easy to construct many other operators with these properties.

\begin{remark}\label{th:remark}
Let $(Y_t)_{t\in\mathbb R}$ be any weakly continuous stochastic process,
with positive increments ($P(Y_s\leq Y_t)=1$ and $P(Y_s<Y_t)>0$ whenever $s<t$) and such that
$\lim_{t\to-\infty}Y_t=-\infty$ (weakly) and $\lim_{t\to\infty}Y_t=\infty$ (weakly).
As an example one may consider the process given by $Y_t=Y+t\cdot Z$,
where $Y$ and $Z$ are any random variables satisfying $P(Z\geq0)=1$ and $P(Z>0)>0$.
Now, let $f$ be an increasing homeomorphism from $\mathbb R$ onto $(0,1)$
and let $g:\mathbb R\to(0,1)$ be the function given by $g(t)=\E f(Y_t)$.
Then $g$ is also an increasing homeomorphism from $\mathbb R$ onto $(0,1)$.
We define the process $(X_x)_{x\in[0,1]}$ as follows:
$X_0=0$, $X_1=1$ and $X_x=f(Y_{g^{-1}(x)})$ for $x\in(0,1)$.
For $x\in[0,1]$ let $\mu_x$ be the distribution of the random variable $X_x$.
The process $(X_x)_{x\in[0,1]}$ is weakly continuous,
it has positive increments (in particular $\mu_x\lst \mu_y$ for every $0\leq x\leq y\leq1$)
and $\E X_x=x$ for every $x\in[0,1]$.
We define the operator $T:\C([0,1])\to\C([0,1])$ by
$T(\varphi)(x)=\E\varphi(X_x)=\int_0^1\varphi(u)\mu_x(du)$.
Then $T$ is a Markov operator satisfying $T(\varphi)=\varphi$ for affine $\varphi$.
Indeed, $T(\varphi)(x)=\E\varphi(X_x)=\varphi(\E X_x)=\varphi(x)$ for $x\in[0,1]$.
Moreover, by Theorem \ref{th:condition}, the operator $T$ satisfies the following I.\ Ra\c{s}a type inequality:
$$2\int_0^1\int_0^1\varphi\left(\tfrac{u+v}2\right)\mu_x(du)\mu_y(dv)
\leq\int_0^1\int_0^1\varphi\left(\tfrac{u+v}2\right)\mu_x(du)\mu_x(dv)
+\int_0^1\int_0^1\varphi\left(\tfrac{u+v}2\right)\mu_y(du)\mu_y(dv)$$
for every convex function $\varphi\in\C([0,1])$ and $x,y\in[0,1]$.
\end{remark}


\section{The Muirhead type inequality for convex orders}

Now we are going to give an inequality (Theorem~\ref{th:main}), which is a generalization of Theorem~\ref{th:condition}.
Before we state the theorem, we need to present two definitions:

\begin{definition}\label{th:definicja}
Let $k\in\mathbb N$ and let $\Pi$ be the set of all permutations of the set $\{1,\dots,k\}$.
We consider the $k$-tuple $(p)=(p_1,\dots,p_k)$ of non-negative integers satisfying $p_1\geq\dots\geq p_k$.
For the given probability distributions $\mu_1,\dots,\mu_k$ and $\pi\in\Pi$, we define $\mu^{(p)}_\pi$ as the following convolution
of probability distributions:
\[\mu^{(p)}_\pi:=(\mu_{\pi(1)})^{*p_1}*(\mu_{\pi(2)})^{*p_2}*\dots*(\mu_{\pi(k)})^{*p_k}\]
We also define
\[\mu^{(p)}:=\tfrac1{k!}\sum_{\pi\in\Pi}\mu^{(p)}_\pi.\]
\end{definition}

Observe, that if we replace $(\mu_1,\dots,\mu_k)$ by any permutation
$(\mu_{\pi(1)},\dots,\mu_{\pi(k)})$, then $\mu^{(p)}$ remains unaltered.
In the set of all the $k$-tuples $(p)$ introduced in Definition \ref{th:definicja}, we consider the following order.
\begin{definition}
We say that $(p)\prec(q)$ if $\sum_{l=1}^kp_l=\sum_{l=1}^kq_l$ and $\sum_{l=1}^mp_l\leq\sum_{l=1}^mq_l$ for $m=1,\dots,k$.
\end{definition}
The above order is a special case of majorization, which has been studied in \cite{Hardy1952} (before Theorem 45),
\cite{MarshallOlkin2011}, and many other sources.

The following condition ($S$) plays an important role.
\begin{definition}
We say that a pair $(p)\prec(q)$ satisfies the condition ($S$), if there exist $1\leq l_1<l_2\leq k $ such that $q_{l_1}=p_{l_1}+1$, $q_{l_2}=p_{l_2}-1$ and $q_l\ =p_l$  for $l\notin\{l_1,l_2\}$.
\end{definition}

\begin{lemma}\label{lm:3.11}
If $(p)\prec(q)$, then there exist $(p)=(p^0)\prec(p^1)\prec\dots\prec(p^I)=(q)$
such that $(p^{i-1})\prec(p^i)$ satisfies ($S$) for $i=1,\dots,I$.
\end{lemma}

\begin{proof}
Aiming for a contradiction, suppose that the conclusion of the lemma does not hold
for some pair $(p)\prec(q)$. Among all such pairs we consider a pair
minimizing the value of $\sum_{m=1}^k\sum_{l=1}^m(q_l-p_l)$
(note that $(p)\prec(q)$ implies that $\sum_{m=1}^k\sum_{l=1}^m(q_l-p_l)$
is a non-negative integer).

We have $(p)\neq(q)$ (otherwise $(p)=(p^0)=(q)$ and $(p)\prec(q)$ satisfies the conclusion of the lemma for $I=0$),
hence at least one of the non-negative numbers $\sum_{l=1}^m(q_l-p_l)$
(with $m=1,\dots,k$) is strictly positive.
Let $1\leq l_1<l_2\leq k$ be such that $\sum_{l=1}^m(q_l-p_l)>0$ for $l_1\leq m<l_2$
and $\sum_{l=1}^{l_1-1}(q_l-p_l)=\sum_{l=1}^{l_2}(q_l-p_l)=0$.
We define $(r)$ as follows: $r_{l_1}=p_{l_1}+1$, $r_{l_2}=p_{l_2}-1$ and $r_l\ =p_l$  for $l\notin\{l_1,l_2\}$.
We will show that $r_1\geq r_2\geq\dots\geq r_k$.
In view of $p_1\geq p_2\geq\dots\geq p_k$, it is enough to show $p_{l_1-1}>p_{l_1}$ (if $l_1>1$) and $p_{l_2}>p_{l_2+1}$ (if $l_2<k$).

If $l_1>1$, then $\sum_{l=1}^{l_1-2}(q_l-p_l)\geq0$, $\sum_{l=1}^{l_1-1}(q_l-p_l)=0$ and $\sum_{l=1}^{l_1}(q_l-p_l)>0$
imply $q_{l_1-1}\leq p_{l_1-1}$ and $q_{l_1}>p_{l_1}$.
Taking into account $q_{l_1-1}\geq q_{l_1}$, we obtain $p_{l_1-1}>p_{l_1}$.

Similarly, if $l_2<k$, then $\sum_{l=1}^{l_2+1}(q_l-p_l)\geq0$, $\sum_{l=1}^{l_2}(q_l-p_l)=0$ and $\sum_{l=1}^{l_2-1}(q_l-p_l)>0$
imply $q_{l_2+1}\geq p_{l_2+1}$ and $q_{l_2}<p_{l_2}$.
Taking into account $q_{l_2}\geq q_{l_2+1}$, we get $p_{l_2}>p_{l_2+1}$.

Now, we will show that $(p)\prec(r)\prec(q)$.
If $l_1\leq m<l_2$, then $\sum_{l=1}^m(r_l-p_l)=1$ and $\sum_{l=1}^m(q_l-r_l)=\sum_{l=1}^m(q_l-p_l)-1\geq0$.
If $m<l_1$ or $m\geq l_2$, then $\sum_{l=1}^m(r_l-p_l)=0$ and $\sum_{l=1}^m(q_l-r_l)=\sum_{l=1}^m(q_l-p_l)\geq0$.
It follows that $(p)\prec(r)\prec(q)$. By the definition of $(r)$ we obtain that ($S$)
holds for $(p)\prec(r)$.

We have
$\sum_{m=1}^k\sum_{l=1}^m(q_l-r_l)=\sum_{m=1}^k\sum_{l=1}^m(q_l-p_l)-(l_2-l_1)<\sum_{m=1}^k\sum_{l=1}^m(q_l-p_l).$
By the minimality of the pair $(p)\prec(q)$, we infer that the pair $(r)\prec(q)$ satisfies the conclusion of the lemma,
hence there exist $(r)=(p^1)\prec(p^2)\prec\dots\prec(p^I)=(q)$
such that $(p^{i-1})\prec(p^i)$ satisfies ($S$) for $i=2,\dots,I$.
Then $(p)=(p^0)\prec(p^1)\prec\dots\prec(p^I)=(q)$ and
$(p^{i-1})\prec(p^i)$ satisfies ($S$) for $i=1,\dots,I$, which is
a contradiction. It follows that the conclusion of the lemma holds for each pair $(p)\prec(q)$.
\end{proof}

\begin{lemma}\label{lm:3.12}
Let $k\in\mathbb N$. Assume that the probability distributions $\mu_1,\dots,\mu_k$ have finite first moments
and they are pairwise comparable in the usual stochastic order. If $(p)\prec(q)$ satisfies condition ($S$),
then $\mu^{(p)}\lcx\mu^{(q)}$.
\end{lemma}

\begin{proof}
Reordering if necessary, we may assume without loss of generality, that $\mu_1\lst\mu_2\lst\dots\lst\mu_k$.
Let $\varphi:\mathbb R\to\mathbb R$ be a convex function.
We need to show
\begin{equation}\label{eq:ghfh}
\tfrac1{k!}\cdot\sum_{\pi\in\Pi}\E\varphi(X^{(p)}_\pi)\leq\tfrac1{k!}\cdot\sum_{\pi\in\Pi}\E\varphi(X^{(q)}_\pi),
\end{equation}
where $X^{(p)}_\pi$ and $X^{(q)}_\pi$ are any random variables satisfying $X^{(p)}_\pi\sim\mu^{(p)}_\pi$
and $X^{(q)}_\pi\sim\mu^{(q)}_\pi$.

We fix $1\leq l_1<l_2\leq k$ from condition ($S$) for $(p)\prec(q)$.
We have $q_{l_1}=p_{l_1}+1>p_{l_1}\geq p_{l_2}>p_{l_2}-1=q_{l_2}$ and $q_l=p_l$ for $l\notin\{l_1,l_2\}$.
We fix any $1\leq u<v\leq k$ and we consider any permutation
$\pi\in\Pi$ satisfying $\pi(l_1)=u$ and $\pi(l_2)=v$.
Let $\pi'\in\Pi$ be given by $\pi'(l_1)=\pi(l_2)=v$, $\pi'(l_2)=\pi(l_1)=u$ and $\pi'(l)=\pi(l)$ for $l\notin\{l_1,l_2\}$.
We define the random variables
$Z_l^s$ (with $l\in\{1,\dots,k\}\setminus\{l_1,l_2\}$ and $1\leq s\leq p_l$), $U_i$ and $V_i$ (with $1\leq i\leq p_{l_1}+p_{l_2}$)
such that
\begin{enumerate}[\upshape (i)]
\item all random variables $Z_l^s$ and random vectors $(U_i,V_i)$ are independent,
\item $Z_l^s\sim\mu_{\pi(l)}$ for $1\leq w\leq k$ and $1\leq s\leq p_{\pi(w)}$,
\item $P(U_i\leq V_i)=1$, $U_i\sim \mu_u$ and $V_i\sim\mu_v$ for $1\leq i\leq p_{l_1}+p_{l_2}$
(here we use our assumption $\mu_u\lst\mu_v$ and Theorem \ref{th:1a1}).
\end{enumerate}
Finally, we put
\[Z:=\sum_{l\notin\{l_1,l_2\}}\sum_{s=1}^{p_l}Z_l^s,\qquad
A:=Z+\sum_{i=1}^{p_{l_1}+1}U_i+\sum_{i=p_{l_1}+2}^{p_{l_1}+p_{l_2}}V_i,
\qquad D:=Z+\sum_{i=1}^{p_{l_2}-1}U_i+\sum_{i=p_{l_2}}^{p_{l_1}+p_{l_2}}V_i\]
\\
\[B:=Z+\sum_{i=1}^{p_{l_1}}U_i+\sum_{i=p_{l_1}+1}^{p_{l_1}+p_{l_2}}V_i,
\qquad C:=Z+\sum_{i=1}^{p_{l_2}-1}U_i+\sum_{i=p_{l_2}}^{p_{l_1}}V_i+U_{p_{l_1}+1}+\sum_{i=p_{l_1}+2}^{p_{l_1}+p_{l_2}}V_i.\]
\\
We have: $B\sim\mu^{(p)}_\pi$,
$C\sim\mu^{(p)}_{\pi'}$,
$A\sim\mu^{(q)}_\pi$,
$D\sim\mu^{(q)}_{\pi'}$ and $A+D=B+C$.
Moreover, $P(A\leq B,C\leq D)=1$.
By Remark \ref{th:HLP} we obtain
$P(\varphi(B)+\varphi(C)\leq\varphi(A)+\varphi(D))=1$, hence
\begin{equation}\label{eq:AKnier}
\E(\varphi(X^{(p)}_\pi)+\varphi(X^{(p)}_{\pi'}))=
\E(\varphi(B)+\varphi(C))\leq
\E(\varphi(A)+\varphi(D))=
\E(\varphi(X^{(q)}_\pi)+\varphi(X^{(q)}_{\pi'})).
\end{equation}
The random variables $B$, $C$, $A$, $D$ play a role, similar to the role that play the random variables $X_1+Y_2$, $X_2+Y_1$, $X_1+X_2$, $Y_1+Y_2$ in the proof of Theorem \ref{th:condition}.
Note that the left side or both sides of \eqref{eq:AKnier} can be equal to $+\infty$.
However, our assumption about finite first moments of $\mu_1,\dots,\mu_k$ ensures
that neither left nor right side of \eqref{eq:AKnier} can be equal to $-\infty$.
Therefore, we may take the sum of the inequalities \eqref{eq:AKnier} over all $\pi\in\Pi$ satisfying $\pi(l_1)=u$, $\pi(l_2)=v$
and over all $u<v$. We obtain
\begin{multline*}
\sum_{\pi\in\Pi}\E\varphi(X^{(p)}_\pi)=
\sum_{1\leq u<v\leq k}\sum_{\substack{\pi\in\Pi\\\pi(l_1)=u\\\pi(l_2)=v}}
\E(\varphi(X^{(p)}_\pi)+\varphi(X^{(p)}_{\pi'}))\leq\\
\sum_{1\leq u<v\leq k}\sum_{\substack{\pi\in\Pi\\\pi(l_1)=u\\\pi(l_2)=v}}
\E(\varphi(X^{(q)}_\pi)+\varphi(X^{(q)}_{\pi'}))=
\sum_{\pi\in\Pi}\E\varphi(X^{(q)}_\pi),
\end{multline*}
\\
which is equivalent to \eqref{eq:ghfh}. The proof is finished.
\end{proof}
As an immediate consequence of Lemmas \ref{lm:3.11} and \ref{lm:3.12}, we obtain our main theorem:
\begin{theorem}\label{th:main}
Let $k\in\mathbb N$. Let $\mu_1,\dots,\mu_k$ be probability distributions with finite first moments ($\int|x|\mu_l(dx)<\infty$ for $l=1,\dots,k$).
If $\mu_1,\dots,\mu_k$ are pairwise comparable in the usual stochastic order
(for each $1\leq i,j\leq k$ we have $\mu_i\lst\mu_j$ or $\mu_i\gst\mu_j$),
then
\[(p)\prec(q)\Longrightarrow\mu^{(p)}\lcx\mu^{(q)}.\]
\end{theorem}

\begin{remark}
Theorem \ref{th:main} is an analogue of Muirhead's inequality (see \cite{Hardy1952}, Theorem 45 or \cite{MarshallOlkin2011}, Section 3G)
with positive numbers replaced by probability distributions, multiplication replaced by convolution, and $\leq$ replaced by $\lcx$.
Moreover, if $x_1,\dots,x_k>0$, then applying Theorem \ref{th:main} with $\mu_l=\delta_{\ln x_l}$ (for $l=1,\dots,k$)
and the convex function $\varphi(x)=e^x$, we obtain the classical Muirhead inequality with integer exponents.
\end{remark}

\begin{example}
If we apply Theorem \ref{th:main}:
\begin{enumerate}[\upshape (i)]
\item
for $k=2$, $(p)=(1,1)$ and $(q)=(2,0)$, then we obtain $\mu*\nu\lcx\frac12(\mu*\mu+\nu*\nu)$
(Theorem \ref{th:condition}),
\item
for $k=m$, $(p)=(1,\dots,1)$ and $(q)=(m,0,\dots,0)$, we get 
$$\mu_1*\dots*\mu_m\lcx\tfrac1m\left[(\mu_1)^{*m}+\dots+(\mu_m)^{*m}\right],$$
which gives the following Ra\c{s}a type inequality (proved in \cite{MRW2017})
\begin{multline*}
 \sum_{i_1,\dots,i_m=0}^n\bigl(b_{n,i_1}(x_1)\cdots b_{n,i_m}(x_1)+\dots+b_{n,i_1}(x_m)\dots b_{n,i_m}(x_m)\\
 -mb_{n,i_1}(x_1)\dots b_{n,i_m}(x_m)\bigr)\varphi\left(\tfrac{i_1+\dots+i_m}{mn}\right)\geq0
\end{multline*}
in the case of $\mu_i=B(n,x_i)$ $(x_i\in[0,1])$ for $i=1,\dots,m$,
\item
for $k=3$, $(p)=(1,1,1)$, $(q)=(2,1,0)$, if $\mu\lst\nu\lst\kappa$, then we get
$$\mu*\nu*\kappa\lcx\tfrac16(\mu*\mu*\nu+\mu*\mu*\kappa+\nu*\nu*\kappa+\nu*\nu*\mu+\kappa*\kappa*\mu+\kappa*\kappa*\nu).$$
\end{enumerate}
Taking binomial, Poisson, negative binomial, gamma, exponential, beta or Gaussian distributions,
as an immediate consequence of Theorem~\ref{th:main} and Lemma~\ref{th:rozklady},
we can obtain several generalizations of the I.\ Ra\c{s}a type inequalities.
\end{example}
One might expect that every polynomial inequality valid for non-negative real numbers has its counterpart
for probability distributions and convex orders. The following example shows that it is very far from true.
\begin{example}
Let $V(x,y)=\frac12x^3y+\frac12xy^3$ and $W(x,y)=\frac18x^4+\frac34x^2y^2+\frac18y^4$.
The polynomials $V$ and $W$ are homogeneous polynomials of degree $4$.
We have $W(x,y)-V(x,y)=\frac18(x-y)^4\geq0$ for every $x,y\in\mathbb R$.
Both $V$ and $W$ have non-negative coefficients and $V(1,1)=W(1,1)=1$.
It follows that $V(\mu,\nu)$ and $W(\mu,\nu)$ are probability distributions
whenever $\mu$ and $\nu$ are probability distributions.
If the expected values (means) $\mathbb E\mu$ and $\mathbb E\nu$ are finite,
then $\mathbb E V(\mu,\nu)=2(\mathbb E\mu+\mathbb E\nu)=\mathbb E W(\mu,\nu)$.
Despite all this regularity the inequality $V(\mu,\nu)\lcx W(\mu,\nu)$
does not need to be valid for $\mu\lst\nu$.
Indeed, let $\mu=\delta_0$ and $\nu=\frac12\delta_0+\frac12\delta_1$.
Then $\mu\lst\nu$, $V(\mu,\nu)=\frac5{16}\delta_0+\frac7{16}\delta_1+\frac3{16}\delta_2+\frac1{16}\delta_3$
and $W(\mu,\nu)=\frac{41}{128}\delta_0+\frac{52}{128}\delta_1+\frac{30}{128}\delta_2+\frac4{128}\delta_3+\frac1{128}\delta_4$.
For the convex function $\varphi(x)=\max(0,x-2)$ we have $\int\varphi(x)V(\mu,\nu)(dx)=\frac1{16}>\frac6{128}=\int\varphi(x)W(\mu,\nu)(dx)$,
hence $V(\mu,\nu)\not\lcx W(\mu,\nu)$.
\end{example}





\bibliographystyle{elsarticle-num}



\end{document}